\documentclass{amsart}

\usepackage{amsmath,amsfonts,amssymb,amsmath,latexsym,multirow}

\newtheorem{theorem}{Theorem}[section]
\newtheorem{lemma}[theorem]{Lemma}
\newtheorem{corollary}[theorem]{Corollary}
\newtheorem{proposition}[theorem]{Proposition}
\newtheorem{remark}[theorem]{Remark}

\theoremstyle{definition}

\numberwithin{equation}{section}

\begin{document}

\title[2-Dimensional Kloosterman Sums Associated with General Linear Groups]{$\begin{array}{c}
         \text{A Recursive Formula for Power Moments of }\\
           \text{2-Dimensional Kloosterman Sums}\\
           \text{Associated with General Linear Groups}
       \end{array}$
}

\author{Dae San Kim and Seung-Hwan Yang}
\address{Department of Mathematics, Sogang University, Seoul 121-742, South Korea}
\curraddr{Department of Mathematics, Sogang University, Seoul
121-742, South Korea} \email{dskim@sogong.ac.kr,
yangdon83@lycos.co.kr}
\thanks{The work was supported by National Foundation of Korea Grant funded by the Korean Government (2009-0072514).}

\subjclass[2000]{}

\date{}

\dedicatory{ }


\begin{abstract}
In this paper, we construct a binary linear code connected with the
Kloosterman sum for $GL(2,q)$. Here $q$ is a power of two. Then we
obtain a recursive formula generating the power moments
2-dimensional Kloosterman sum, equivalently that generating the even
power moments of Kloosterman sum in terms of the frequencies of
weights in the code. This is done via Pless power moment identity
and by utilizing the explicit expression of the Kloosterman sum for
$GL(2,q)$.

Index terms - recursive formula, power moment, Kloosterman sum,
2-dimensional Kloosterman sum, general linear group, Pless power
moment identity, weight distribution.
%
\end{abstract}

\maketitle
\section{Introduction}
Let $\psi$ be a nontrivial additive character of the finite field
$\mathbb{F}_q$ with $q=p^r$ elements ($p$ a prime), and let $m$ be a
positive integer. Then the $m$-dimensional Kloosterman sum
$K_m(\psi;a)$(\cite{RH}) is defined by
\begin{equation*}
K_m(\psi;a)=\sum_{\alpha_1,\cdots,\alpha_m \in \mathbb{F}_q^*}
\psi(\alpha_1+\cdots+\alpha_m+a\alpha_1^{-1}\cdots\alpha_m^{-1})~(a\in
\mathbb{F}_q^*).
\end{equation*}
In particular, if $m=1$, then $K_1(\psi;a)$ is simply denoted by
$K(\psi;a)$, and is called the Kloosterman sum. For this, we have
the Weil bound(cf. \cite{RH})
\begin{equation}\label{s1}
|K(\lambda;a)|\leq 2\sqrt{q}.
\end{equation}
The Kloosterman sum was introduced in 1926(\cite{K1}) to give an
estimate for the Fourier coefficients of modular forms.

\noindent For each nonnegative integer $h$, we will denote the
$h$-th moment of the $m$-dimensional Kloosterman sum $K_m(\psi;a)$
by $MK_m(\psi)^h$. Namely, it is given by

\begin{equation*}
MK_m(\psi)^h=\sum_{a\in\mathbb{F}_q^*}K_m(\psi;a)^h.
\end{equation*}
If $\psi=\lambda$ is the canonical additive character of
$\mathbb{F}_q$, then $MK_m(\lambda)^h$ will be simply denoted by
$MK^h_m$. If further $m=1$, for brevity $MK^h_1$ will be indicated
by $MK^h$.

Explicit computations on power moments of Kloosterman sums were
begun with the paper \cite{HS} of Sali$\acute{e}$ in 1931, where he
showed, for any odd prime $q$,
\begin{equation*}
MK^h=q^2M_{h-1}-(q-1)^{h-1}+2(-1)^{h-1}~(h\geq1).
\end{equation*}
Here $M_0=0$, and for $h\in\mathbb{Z}_{>0}$,
\begin{equation*}
M_h=\Big|\Big\{(\alpha_1,\cdots,\alpha_h)\in(\mathbb{F}^*_q)^h|
\sum_{j=1}^{h}\alpha_j=1=\sum_{j=1}^{h}\alpha_{j}^{-1}\Big\}\Big|.
\end{equation*}
For $q=p$ odd prime, Sali$\acute{e}$ obtained $MK^1$, $MK^2$,
$MK^3$, $MK^4$ in that same paper by determining $M_1$, $M_2$,
$M_3$.

From now on, let us assume that $q=2^r$. Carlitz\cite{L2} evaluated
$MK^h$ for $h\leq4$. Moisio was able to find explicit expressions of
$MK^h$, for $h\leq10$ (cf. \cite{M1}). This was done, via Pless
power moment identity, by connecting moments of Kloosterman sums and
the frequencies of weights in the binary Zetterberg code of length
$q+1$, which were known by the work of
 Schoof and Vlugt in \cite{RM}.

Simple identities involving Kloosterman sums, multi-dimensional
Kloosterman sums and powers of Kloosterman sums were used in
\cite{DS4} and \cite{DS5} to construct binary linear codes
associated with them. And then they were used to obtain recursive
formulas generating power moments of Kloosterma sums, $m$-multiple
power moments of Kloosterman sums and power moments of
multi-dimensional Kloosterman sums.

In this paper, along the line of the previous papers \cite{DS4} and
\cite{DS5}, we will utilize one simple identity connecting the
Kloosterman sum for $GL(2,q)$ and the ordinary Kloosterman sum(cf.
(\ref{s8})). Then we will be able to produce a recursive formula
generating the power moments of 2-dimensional Kloosterman sums,
equivalently that generating the even power moments of Kloosterman
sums. To do that, we construct a binary linear code connected with
the Kloosterman sum for $GL(2,q)$.

Theorem \ref{1} of the following(cf. (\ref{s2})-(\ref{s4})) is the
main result of this paper. Henceforth, we agree that the binomial
coefficient $\binom{b}{a}=0$, if $a>b$ or $a<0$.

\begin{theorem}\label{1}
Let $q=2^r$. Then we have the following:

\noindent $($a$)$ For $r\geq2$, and $h=1,2,\cdots,$
\begin{equation}\label{s2}
\begin{split}
MK_2^{h}
&=\sum_{l=0}^{h-1}(-1)^{h+l+1}\binom{h}{l}(q^3-2q^2-q+1)^{h-l} MK_2^{l}\\
&\quad+q^{1-h}\sum_{j=0}^{min\{N,h\}}(-1)^{h+j} C_j \sum_{t=j}^{h}
t! S(h,t)2^{h-t}\binom{N-j}{N-t},
\end{split}
\end{equation}

\noindent $($b$)$ For $r\geq2$, and $h=1,2,\cdots,$
\begin{equation}\label{s3}
\begin{split}
MK^{2h}
&=\sum_{l=0}^{h-1}(-1)^{h+l+1} \binom{h}{l}(q^3-2q^2+1)^{h-l}MK^{2l}\\
&\quad+q^{1-h}\sum_{j=0}^{min\{N,h\}}(-1)^{h+j} C_j \sum_{t=j}^{h}
t! S(h,t)2^{h-t}\binom{N-j}{N-t},
\end{split}
\end{equation}
where $N=|GL(2,q)|=q(q-1)(q^2-1)$, and $\{C_j\}^N_{j=0}$ is the
weight distribution of $C(GL$$(2,q))$ given by
\begin{equation}\label{s4}
C_j=\sum\binom{m_0}{\nu_0}\prod_{|t|<2\sqrt{q},~t\equiv-1(4)}\prod_{K(\lambda;\beta^{-1})=t}
\binom{m_t}{\nu_\beta}~(j=0,\cdots,N),
\end{equation}
with the sum running over all the sets of nonnegative integers
$\{\nu_\beta\}_{\beta\in\mathbb{F}_q}$ satisfying
\begin{equation*}
\sum_{\beta\in\mathbb{F}_q}\nu_\beta=j~ and~
\sum_{\beta\in\mathbb{F}_q}\nu_\beta\beta=0,
\end{equation*}
\begin{equation*}
m_0=q(2q^2-2q-1),
\end{equation*}
and
\begin{equation*}
m_t=q(q^2-2q-1+t),
\end{equation*}
for all integers $t$ satisfying $|t|<2\sqrt{q}$ and $t\equiv-1(mod~
4)$.

In addition, $S(h,t)$ is the Stirling number of the second kind
given by
\begin{equation}\label{s5}
S(h,t)= \frac{1}{t!} \sum_{j=0}^{t}(-1)^{t-j}\binom{t}{j} j^{h}.
\end{equation}
\end{theorem}

\section{Preliminaries}
Throughout this paper, the following notations will be used:

\begin{equation*}
\begin{split}
q&=2^r~(r\in\mathbb{Z}_{>0}),\\
\mathbb{F}_q&=\text{the finite field with $q$
 elements},\\
tr(x)&=x+x^2+\cdots+x^{2^{r-1}}~ \text{the trace
 function} ~\mathbb{F}_q \rightarrow \mathbb{F}_2,\\
\lambda(x)&=(-1)^{tr(x)}~ \text{the canonical additive
 character of}~ \mathbb{F}_q.
\end{split}
\end{equation*}
Then any nontrivial additive character $\psi$ of $\mathbb{F}_q$ is
given by $\psi(x)=\lambda(ax)$, for a unique $a\in \mathbb{F}_q^*$.

For any nontrivial additive character $\psi$ of $\mathbb{F}_q$ and
$a\in\mathbb{F}_q^*$, the Kloosterman sum $K_{GL(t,q)}(\psi;a)$ for
$GL(t,q)$ is defined as
\begin{equation*}
K_{GL(t,q)}(\psi;a)=\sum_{g\in GL(t,q)} \psi(Tr{g}+aTr{g^{-1}}).
\end{equation*}

\noindent Observe that, for $t=1$,$K_{GL(1,q)}(\psi;a)$ denotes the
Kloosterman sum $K(\psi;a)$.

In \cite{DS3}, it is shown that $K_{GL(t,q)}(\psi;a)$ satisfies the
following recursive relation: \noindent for integers $t \geq 2$,
$a\in\mathbb{F}_q^*$,
\begin{equation}\label{s6}
\begin{split}
K_{GL(t,q)}(\psi;a)=q^{t-1}K_{GL(t-1,q)}&(\psi;a)K(\psi;a)\\
&+q^{2t-2}(q^{t-1}-1)K_{GL(t-2,q)}(\psi;a),
\end{split}
\end{equation}
where we understand that $K_{GL(0,q)}(\psi;a)=1$.

\begin{theorem}\label{2}
$($\cite{L1}$)$ For the canonical additive character $\lambda$ of
$\mathbb{F}_q$, and $a\in\mathbb{F}_q^*$,
\begin{equation}\label{s7}
K_2(\lambda;a)=K(\lambda;a)^2-q.
\end{equation}
\end{theorem}

Our paper will be based on the $t=2$ case of the identity in
(\ref{s6}).

\begin{proposition}\label{3}
For the canonical additive character $\lambda$ of $\mathbb{F}_q$, we
have:
\begin{equation}\label{s8}
K_{GL(2,q)}(\lambda;a)=qK(\lambda;a)^2+q^2(q-1)=qK_2(\lambda;a)+q^3.
\end{equation}
\end{proposition}

\begin{proposition}\label{4}$($\cite{DS2}$)$
For $n=2^s~(s\in\mathbb{Z}_{\geq0})$, $\lambda$ the canonical
additive character of $~\mathbb{F}_q$, and $~a\in\mathbb{F}_q^*$,
\begin{equation}\label{s9}
K(\lambda;a^n)=K(\lambda;a).
\end{equation}
\end{proposition}

\begin{remark}\label{5}
In fact, (\ref{s9}) holds more generally for multi-dimensional
Kloosterman sums. For $n=2^s~(s\in\mathbb{Z}_{\geq0})$, $\lambda$
the canonical additive character of $\mathbb{F}_q$,
$a\in\mathbb{F}_q^*$, and any positive integer $m$,
\begin{equation}\label{s10}
K_m(\lambda;a^n)=K_m(\lambda;a).
\end{equation}
\end{remark}

\noindent The order of the general linear group $GL(n,q)$ is given
by
\begin{equation}\label{s11}
g_n=\prod_{j=0}^{n-1}(q^n-q^j)=q^{\binom{n}{2}}\prod_{j=1}^{n}(q^j-1).
\end{equation}

\section{Construction of codes}

Let
\begin{equation}\label{s12}
N=|GL(2,q)|=q(q-1)(q^2-1).
\end{equation}

Here we will construct a binary linear code $C$ of length $N$
connected with the Kloosterman sum for $GL(2,q)$. Let
$g_1$,$\cdots$,$g_N$ be a fixed ordering of the elements in
$GL(2,q)$, and let
\begin{equation*}
v=( Tr{g_1}+Tr{g_1}^{-1},\;\; \cdots,\;\;Tr{g_N}+Tr{g_N}^{-1}) \in
\mathbb{F}_{q}^{N}.
\end{equation*}
The binary linear code $C=C(GL(2,q))$is defined as
\begin{equation}\label{s13}
C=\{u\in \mathbb{F}_2^N~|~u\cdot v=0\}.
\end{equation}
The following Delsarte's theorem is well-known.

\begin{theorem}\label{6} $($\cite{FN}$)$
Let $B$ be a linear code over $\mathbb{F}_q$. Then
\begin{equation*}
(B|_{\mathbb{F}_2})^\bot=tr(B^\bot).
\end{equation*}
\end{theorem}
\noindent In view of this theorem, the dual $C^\bot$ of $C$ is given
by
\begin{equation}\label{s14}
 C^\bot=\{c(a)=(tr(a(Tr{g_1}+Tr{g_1}^{-1})),\cdots,tr(a(Tr{g_N}+
 Tr{g_N}^{-1})))~|~a\in\mathbb{F}_q\}.
\end{equation}

The following estimate is very coarse but will serve for our
purpose.

\begin{lemma}\label{7}
For any $a\in\mathbb{F}_q^*$, and $\psi$ any nontrivial additive
character of $\mathbb{F}_q$,
\begin{equation}\label{s15}
|K_{GL(n,q)}(\psi;a)|<|GL(n,q)|,~\text{for $n\geq2$ and $q\geq4$},
\end{equation}
and
\begin{equation*}
|K_{GL(1,q)}(\psi;a)|<|GL(1,q)|,~\text{for $q\geq8$}.
\end{equation*}
\end{lemma}
\begin{proof}
For $n=1$, this is trivial, since $2\sqrt{q}<q-1$, for $q\geq8$. For
$n=2$, from (\ref{s6})
\begin{equation}\label{s16}
K_{GL(2,q)}(\psi;a)=qK(\psi;a)^2+q^2(q-1),
\end{equation}
and hence from (\ref{s1}) and (\ref{s16}), for $q\geq4$,
\begin{equation}\label{s17}
|K_{GL(2,q)}(\psi;a)|\leq q^3+3q^2<q(q-1)(q^2-1)=|GL(2,q)|.
\end{equation}
For $n=3$, from (\ref{s6}),
\begin{equation}\label{s18}
K_{GL(3,q)}(\psi;a)=q^2K_{GL(2,q)}(\psi;a)K(\psi;a)+q^4(q^2-1)K(\psi;a),
\end{equation}
and hence from (\ref{s1}), (\ref{s17}), and (\ref{18}), for
$q\geq4$,
\begin{equation*}
|K_{GL(3,q)}(\psi;a)|<2q^{\frac{7}{2}}(q^2-1)(2q-1)<q^3(q-1)(q^2-1)(q^3-1)=|GL(3,q)|.
\end{equation*}
Assume now that $n\geq4$ and that (\ref{s15}) holds for all integers
less than $n$ and greater than and equal to 2, for $q\geq4$. Then,
from (\ref{s1}), (\ref{s6}), and (\ref{s11}), and for $q\geq4$,
\begin{equation*}
|K_{GL(n,q)}(\psi;a)|<q^{\binom{n}{2}}(q+2\sqrt{q})\prod_{j=1}^{n-1}(q^j-1)<
q^{\binom{n}{2}}\prod_{j=1}^{n}(q^j-1)<|GL(n,q)|.
\end{equation*}
\end{proof}

\begin{remark}\label{8}
It was shown in $[3, Theorem~2]$ that, for any nontrivial additive
character $\psi$ of $\mathbb{F}_q$ and $a\in\mathbb{F}_q^*$,
\begin{equation*}
\begin{split}
K_{GL(n,q)}(\psi;a^2)&=\sum_{g\in GL(n,q)}\psi(a(Tr{g}+Tr{g^{-1}}))\\
&=(-1)^nq^{\binom{n}{2}}\sum_{j=0}^{n}\begin{bmatrix}
                                                              n \\
                                                              j \\
                                                            \end{bmatrix}_q
                                                            \omega^j\bar{\omega}^{n-j},
\end{split}
\end{equation*}
where $\omega$, $\bar{\omega}$ are complex numbers, depending on
$\psi$ and $a$, with $|\omega|$=$|\bar{\omega}|$=$\sqrt{q}$. Thus
\begin{equation*}
|K_{GL(n,q)}(\psi;a^2)|\leq
q^{\frac{1}{2}n^2}\sum_{j=0}^{n}\begin{bmatrix}
                                                              n \\
                                                              j \\
                                                            \end{bmatrix}_q,
\end{equation*}
and, in particular, we get
\begin{equation*}
|K_{GL(2,q)}(\psi;a^2)|\leq q^{2}\sum_{j=0}^{2}\begin{bmatrix}
                                                              2 \\
                                                              j \\
                                                            \end{bmatrix}_q
                                                            =q^2(q+3).
\end{equation*}
\end{remark}

\begin{proposition}\label{9}
The map $\mathbb{F}_q \rightarrow C^\bot~(a\mapsto c(a))$ is an
$\mathbb{F}_2$-linear isomorphism for $q\geq4$.
\end{proposition}
\begin{proof}
The map is clearly $\mathbb{F}_2$-linear and surjective. Let $a$ be
in the kernel of the map. Then $tr(a(Tr{g}+Tr{g^{-1}}))=0$, for all
$g\in GL(2,q)$. Suppose that $a\neq 0$. Then, on the one hand,
\begin{equation}\label{s19}
\begin{split}
|GL(2,q)|&=\sum_{g\in GL(2,q)}(-1)^{tr(a(Tr{g}+Tr{g^{-1}}))}\\
&=\sum_{g\in GL(2,q)}\lambda(a(Tr{g}+Tr{g^{-1}}))\\
&=\sum_{g\in GL(2,q)}\lambda(Tr{g}+a^2Tr{g^{-1}})~(g\rightarrow a^{-1}g)\\
&=K_{GL(2,q)}(\lambda;a^2).
\end{split}
\end{equation}
As $q\geq4$, (\ref{s19}) is on the other hand strictly less than
$|GL(2,q)|$ by Lemma \ref{7}. This is a contradiction. So we must
have $a=0$.
\end{proof}

\begin{remark}\label{10}
$($a$)$ If $q=2$, one checks easily that the kernel of the map
$\mathbb{F}_2 \rightarrow C^\bot$ is $\mathbb{F}_2$.

\noindent $($b$)$ The fact that the map in Proposition \ref{9} is
injective follows also from (\ref{s1}) and (\ref{s22}), since they
imply that $n(\beta)>0$, for all $\beta$, provided that $q\geq4$.
\end{remark}

\begin{proposition}\label{11} $($\cite{DS2}$)$
Let $\lambda$ be the canonical additive character of $\mathbb{F}_q$,
$m\in\mathbb{Z}_{>0}$, $\beta\in\mathbb{F}_q$. Then
\begin{equation}\label{s20}
\begin{split}
&\sum_{a\in\mathbb{F}_q^*}\lambda(-a\beta)K_m(\lambda;a)\\
&~=
\begin{cases}
qK_{m-1}(\lambda;\beta^{-1})+(-1)^{m+1}, & \text{if $\beta\neq 0$},\\
(-1)^{m+1},& \text{if $\beta=0$.}
\end{cases}
\end{split}
\end{equation}
with the convention $K_0(\lambda;\beta^{-1})=\lambda(\beta^{-1})$.
\end{proposition}

Let
\begin{equation}\label{s21}
n(\beta)=|\{g\in GL(2,q)|Tr{g}+Tr{g^{-1}}=\beta\}|.
\end{equation}

Then, with $N$ as in (\ref{s12}),
\begin{equation*}
\begin{split}
qn(\beta)&=N+\sum_{\alpha\in\mathbb{F}_q^*}\lambda(-\alpha\beta)\sum_{g\in
GL(2,q)}\lambda(\alpha(Tr{g}+Tr{g^{-1}}))\\
&=N+\sum_{\alpha\in\mathbb{F}_q^*}\lambda(-\alpha\beta)
K_{GL(2,q)}(\lambda;\alpha^2)\\
&=N+\sum_{\alpha\in\mathbb{F}_q^*}\lambda(-\alpha\beta)
(qK_2(\lambda;\alpha^2)+q^3)(cf. (\ref{s8}))\\
&=N+q\sum_{\alpha\in\mathbb{F}_q^*}\lambda(-\alpha\beta)
K_2(\lambda;\alpha^2)+q^3\sum_{\alpha\in\mathbb{F}_q^*}\lambda(-\alpha\beta)\\\
&=N+q\sum_{\alpha\in\mathbb{F}_q^*}\lambda(-\alpha\beta)
K_2(\lambda;\alpha)+q^3\sum_{\alpha\in\mathbb{F}_q^*}\lambda(-\alpha\beta).(cf.
(\ref{s10}))
\end{split}
\end{equation*}

Now, from Proposition \ref{11}, we obtain the following.

\begin{proposition}\label{12}
Let $n(\beta)$ be as in (\ref{s21}). Then we have
\begin{equation}\label{s22}
n(\beta)=
\begin{cases}
q(q^2-2q-1+K(\lambda;\beta^{-1})), & \text{if $\beta\neq 0$},\\
q(2q^2-2q-1),& \text{if $\beta=0$.}
\end{cases}
\end{equation}
\end{proposition}

\section{Power moments of 2-dimensional Kloosterman sums}
In this section, we will be able to find, via Pless power moment
identity, a recursive formula for the power moments of 2-dimensional
Kloosterman sums or equivalently for the even power moments of
Kloosterman sums in terms of the frequencies of weights in
$C=C(GL(2,q))$.

\begin{theorem}\label{15}\rm(Pless power moment identity\rm):
Let $B$ be a $q$-ary $[n,k]$ code, and let $B_{i}$ (resp. $B_{i}
^{\bot})$ denote the number of codewords of weight $i$ in $B$ (resp.
in $B^{\bot})$. Then, for $h=0, 1, 2, \cdots$,

\begin{equation}\label{s23}
\sum_{j=0}^{n}j^{h}B_{j}=\sum_{j=0}^{min\{n,h\}}(-1)^{j}B_{j}^{\bot}
\sum_{t=j}^{h} t! S(h,t)q^{k-t}(q-1)^{t-j}\binom{n-j}{n-t},
\end{equation}
where $S(h,t)$ is the Stirling number of the second kind defined in
(\ref{s5}).
\end{theorem}

From now on, we will assume that $q\geq4(i.e.,r\geq2)$, so that
every codeword in $C(GL(2,q))^\bot$ can be written as $c(a)$, for a
unique $a\in\mathbb{F}_q$(cf. Proposition \ref{9}). This also allows
one to use Theorem \ref{17}.

\begin{lemma}\label{14}
Let $c(a)=(tr(a(Tr{g_1}+Tr{g_1}^{-1})),\cdots,
 tr(a(Tr{g_N}+Tr{g_N}^{-1})))\in C(GL(2,q))^\bot$, for
 $a\in\mathbb{F}_q^*$. Then the Hamming weight $w(c(a))$ can be
 expressed as follows:
\begin{equation}\label{s24}
w(c(a))=\frac{1}{2}q(q^3-2q^2+1-K(\lambda;a)^2)
\end{equation}
\begin{equation}\label{s25}
\quad\quad\quad\quad\quad=\frac{1}{2}q(q^3-2q^2-q+1-K_2(\lambda;a)).
\end{equation}
\end{lemma}
\begin{proof}
\begin{equation*}
\begin{split}
w(c(a))&=\frac{1}{2}\sum_{i=1}^{N}(1-(-1)^{tr(a(Tr{g_i}+Tr{g_i^{-1}}))})\\
&=\frac{1}{2}(N-\sum_{g\in GL(2,q)}\lambda(a(Tr{g}+Tr{g^{-1}})))\\
&=\frac{1}{2}(N-\sum_{g\in GL(2,q)}\lambda(Tr{g}+a^2Tr{g^{-1}}))\\
&=\frac{1}{2}(N-K_{GL(2,q)}(\lambda;a^2))\\
&=\frac{1}{2}(N-qK(\lambda;a)^2-q^2(q-1))~(cf.(\ref{s8}),~(\ref{s9}))\\
&=\frac{1}{2}q(q^3-2q^2+1-K(\lambda;a)^2)~(cf.(\ref{s12}))\\
&=\frac{1}{2}q(q^3-2q^2-q+1-K_2(\lambda;a)).~(cf.(\ref{s7}))
\end{split}
\end{equation*}
\end{proof}

Let $u=(u_1,\cdots,u_N)\in\mathbb{F}_2^N$, with $\nu_\beta$ 1's in
the coordinate places where $Tr{g_j}+Tr{g_j}^{-1}=\beta$, for each
$\beta\in\mathbb{F}_q$. Then we see from the definition of the code
$C(GL(2,q))$(cf. (\ref{s13})) that $u$ is a codeword with weight $j$
if and only if $\sum_{\beta\in\mathbb{F}_q}\nu_\beta=j$ and
$\sum_{\beta\in\mathbb{F}_q}\nu_\beta\beta=0$(an identity in
$\mathbb{F}_q$). As there are
$\prod_{\beta\in\mathbb{F}_q}\binom{n(\beta)}{\nu_\beta}$ many such
codewords with weight $j$,  we obtain the following result.

\begin{proposition}\label{15}
Let $\{C_j\}^N_{j=0}$ be the weight distribution of $C(GL$$(2,q))$,
where $C_j$ denotes the frequency of the codewords with weight $j$
in $C$. Then
\begin{equation}\label{s26}
C_j=\sum\prod_{\beta\in\mathbb{F}_q} \binom{n(\beta)}{\nu\beta},
\end{equation}
where the sum runs over all the sets of nonnegative integers
$\{\nu_\beta\}_{\beta\in\mathbb{F}_q}~(0\leq\nu_{\beta}\leq
n(\beta))$, satisfying
\begin{equation}\label{s27}
\sum_{\beta\in\mathbb{F}_q}\nu_\beta=j~ \text{and}~
\sum_{\beta\in\mathbb{F}_q}\nu_\beta\beta=0.
\end{equation}
\end{proposition}

\begin{corollary}\label{16}
Let $\{C_j\}_{j=0}^N$ be the weight distribution of $C(GL(2,q))$.
Then we have: $C_j=C_{N-j}$, for all $j$, with $0\leq j\leq N$.
\end{corollary}
\begin{proof}
Under the replacements $\nu_{\beta}\rightarrow n(\beta)-\nu_\beta$,
for each $\beta\in\mathbb{F}_q$, the first equation in (\ref{s27})
is changed to $N-j$, while the second one in (\ref{s27}) and the
summands in (\ref{s26}) are left unchanged. Here the second sum in
(\ref{s27}) is left unchanged, since
$\sum_{\beta\in\mathbb{F}_q}n(\beta)\beta=0$, as one can see by
using the explicit expression of $n(\beta)$ in (\ref{s22}).
\end{proof}

\begin{theorem}\label{17} $($\cite{GJ}$)$
Let $q=2^r$, with $r\geq 2$. Then the range $R$ of $K(\lambda;a)$,
as a varies over $\mathbb{F}_q^*$, is given by
\begin{equation*}
R=\{t\in\mathbb{Z}~|~|t|<2\sqrt{q},~t\equiv-1(mod~4)\}.
\end{equation*}

In addition, each value $t\in R$ is attained exactly $H(t^2-q)$
times, where $H(d)$ is the Kronecker class number of $d$.
\end{theorem}

Now, we get the following formula in (\ref{s28}), by applying the
formula in (\ref{s26}) to $C(GL(2,q))$, using the explicit values of
$n(\beta)$ in (\ref{s22}) and taking Theorem \ref{17} into
consideration.

\begin{theorem}\label{18}
Let $\{C_j\}^N_{j=0}$ be the weight distribution of $C(GL$$(2,q))$.
Then
\begin{equation}\label{s28}
C_j=\sum\binom{m_0}{\nu_0}\prod_{|t|<2\sqrt{q},~t\equiv-1(4)}\prod_{K(\lambda;\beta^{-1})=t}
\binom{m_t}{\nu_\beta}~(j=0,\cdots,N),
\end{equation}
where the sum is over all the sets of nonnegative integers
$\{\nu_\beta\}_{\beta\in\mathbb{F}_q}$ satisfying
$\sum_{\beta\in\mathbb{F}_q}\nu_\beta=j$ and
$\sum_{\beta\in\mathbb{F}_q}\nu_\beta\beta=0$,
\begin{equation*}
m_0=q(2q^2-2q-1),
\end{equation*}
and
\begin{equation*}
m_t=q(q^2-2q-1+t),
\end{equation*}
for all integers $t$ satisfying $|t|<2\sqrt{q}$ and $t\equiv-1(mod~
4)$.
\end{theorem}

We now apply the Pless power moment identity in (\ref{s23}) to
$C(GL(2,q))^{\bot}$, in order to obtain the results in Theorem
\ref{1} (cf. (\ref{s2})-(\ref{s4})) about recursive  formulas.

Then the left hand side of that identity in (\ref{s23}) is equal to
\begin{equation}\label{s29}
\sum_{a\in\mathbb{F}_{q}^{*}}w(c(a))^h,
\end{equation}
with the $w(c(a))$ given either by (\ref{s24}) or by (\ref{s25}).

\noindent Using the expression of $w(c(a))$ in (\ref{s25}),
(\ref{s29}) is

\begin{equation}\label{s30}
\begin{split}
&(\frac{q}{2})^h\sum_{a \in
\mathbb{F}_{q}^{*}}(q^3-2q^2-q+1-K_2(\lambda;a))^h\\
=&(\frac{q}{2})^h\sum_{a\in\mathbb{F}_q^*}\sum_{l=0}^{h}(-1)^{l}\binom{h}{l}
(q^3-2q^2-q+1)^{h-l} K_2(\lambda;a)^{l}\\
=&(\frac{q}{2})^h\sum_{l=0}^{h}(-1)^{l}\binom{h}{l}
(q^3-2q^2-q+1)^{h-l} MK_2^{l}.
\end{split}
\end{equation}

\noindent Equivalently, using the expression of $w(c(a))$ in
(\ref{s24}), (\ref{s29}) is

\begin{equation}\label{s31}
(\frac{q}{2})^h\sum_{l=0}^{h}(-1)^{l}\binom{h}{l} (q^3-2q^2+1)^{h-l}
MK^{2l}.
\end{equation}

\noindent On the other hand, the right hand side of the identity in
(\ref{s23}) is
\begin{equation}\label{s32}
q\sum_{j=0}^{min\{N,h\}}(-1)^j C_j \sum_{t=j}^{h} t!
S(h,t)2^{-t}\binom{N-j}{N-t}.
\end{equation}

Our main results in Theorem \ref{1} (cf. (\ref{s2})-(\ref{s4})) now
follow by equating (\ref{s30}) and (\ref{s32}), and (\ref{s31}) and
(\ref{s32}). Also, one has to separate the term corresponding to
$l=h$ in (\ref{s30}) and (\ref{s31}), and note
$dim_{\mathbb{F}_2}C(GL(2,q))=r$.

Note here that, in view of (\ref{s7}), obtaining power moments of
2-dimensional Kloosterman sums is equivalent to getting even power
moments of Kloosterman sums.

\bibliographystyle{amsplain}

\end{document}